\documentclass[12pt]{amsart}

\usepackage{amssymb}

\usepackage{enumerate}

\usepackage{graphicx}

\makeatletter
\@namedef{subjclassname@2010}{%
  \textup{2010} Mathematics Subject Classification}
\makeatother

\newtheorem{thm}{Theorem}[section]

\newtheorem{lemma}[thm]{Lemma}

\theoremstyle{definition}

\def\Z{\mathbb Z}
\def\N{\mathbb N}

\def\R{\mathbb R}

\def\pmod #1{\ ({\rm mod}\ #1)}
\def\le{\leqslant}
\def\ge{\geqslant}

\numberwithin{equation}{section}

\begin{document}


\baselineskip=17pt



\title[Squares and Powers of 2]{Four squares of primes and powers of 2}

\author[L. Zhao]{Lilu  Zhao}
\address{School of Mathematics, Hefei University of Technology, Heifei 230009, People's Republic of China}
\email{zhaolilu@gmail.com}

\date{}

\begin{abstract}
By developing the method of Wooley on the quadratic
Waring-Goldbach problem, we prove that all sufficiently large even
integers can be expressed as a sum of four squares of primes and
$46$ powers of $2$.
\end{abstract}



\maketitle

\section{Introduction}
In 1950's, it was shown by Linnik \cite{Lin1,Lin2} that every
sufficiently large integer can be represented as the sum of two
primes and $K$ powers of two, where $K$ is an absolute number. In
1975, Gallagher \cite{G} obtained a stronger result via a
different approach. An explicit value for the number $K$ was
firstly obtained by Liu, Liu and Wang \cite{LLW1}, who established
that $K=54000$ is acceptable. The value for the number $K$ was
subsequently improved by Li \cite{Li00}, Wang \cite{W1999} and Li
\cite{Li01}. Recently, a rather different method was described by
Heath-Brown and Puchta \cite{HP}, and independently by Pintz and
Ruzsa \cite{PR}. In particular, it was shown in \cite{HP} that
$K=13$ is acceptable, and it was claimed in \cite{PR} that $K=8$
is acceptable.

In 1938, Hua \cite{H} proved that all large integers congruent to
$5$ modulo $24$ can be represented as the sum of five squares of
primes. It seems reasonable to conjecture that every large integer
congruent to $4$ modulo $24$ can be expressed as the sum of four
squares of primes. This problem is still open while Br\"{u}dern
and Fouvry \cite{BF} established that every sufficiently large
integer $n\equiv 4\pmod {24}$ is the sum of four squares of almost
primes.

In 1999, Liu, Liu and Zhan \cite{LLZ} investigated the expression
\begin{equation}\label{e0} N=p_1^2+p_2^2+p_3^2+p_4^2+2^{\nu_1}+\cdots+2^{\nu_k},\end{equation}
and proved that every sufficiently large even integer can be
represented as the sum of four squares of primes and $k$ powers of
two. It was shown in \cite{LiuLiu} that $k=8330$ is acceptable.
This value was sharpened to $k=165$ in \cite{LiuLv} and $k=151$ in
\cite{Li1}. The purpose of this paper is to establish the
following result.
\begin{thm}\label{thm}
 Every sufficiently large even integer can be represented as a
sum of four squares of primes and 46 powers of $2$.
\end{thm}

We establish Theorem \ref{thm} by means of the Hardy-Littlewood
method in combination with the linear sieve. In order to bound the
contributions of the minor arcs, in previous works
\cite{LiuLv,Li1}, one may encounter the integral of the type
$\int_0^1|T(\alpha)G(\alpha)|^4d\alpha$, where $T(\alpha)$ and
$G(\alpha)$ are defined in \eqref{pro}. The above integral is no
more than the number of solutions for $p_1^2+p_2^2-p_3^2-p_4^2=h$
with $h=2^{\nu_1}+2^{\nu_2}-2^{\nu_3}-2^{\nu_4}$, where $p_j^2\le
N$ and $\nu_i\le L$. The contribution from $h=0$ can be obtained
by Rieger's result \cite{R}. Then as was pointed out in
\cite{LiuLv}, a crucial step is to bound from above the number of
solutions of the equation $p_1^2+p_2^2-p_3^2-p_4^2=h$ with nonzero
$h$. The machinery of Br\"{u}dern and Fouvry was employed directly
to provide such estimation, while the information on the powers of
two was lost in this process. Our approach is different from
theirs. Instead of the integral
$\int_0^1|T(\alpha)G(\alpha)|^4d\alpha$, we investigate a new
integral $\int_0^1|T(\alpha)^4G(\alpha)^{14}|d\alpha$. Now the
loss is that we need more variables for the powers of $2$ in the
mean value integral, while the gain is a situation where we can
apply a linear sieve procedure to the equation involving four
squares of primes and fourteen powers of two. This approach is
motivated by the works of Wooley \cite{Wooley} and of Tolev
\cite{Tolev}. In view of \cite{BF,HT,Wooley}, it seems very hard
to solve the equation $p_1^2+p_2^2-p_3^2-x^2=h$ for nonzero $h$,
while Wooloy's argument works well to establish the asymptotic
formula for the number of solutions for the equation
$p_1^2+p_2^2-p_3^2-x^2+\sum_{j=1}^3(2^{u_j}-2^{v_j})=0$ in a
suitable box, where $x$ is a natural number. Motivated by Wooley's
result, Tolev considered the exceptional set for the equation
$p_1^2+p_2^2+p_3^2+x^2=n$ with $x$ an almost prime, and his
argument works for the equation
$p_1^2+p_2^2-p_3^2-(dx)^2+\sum_{j=1}^{t}(2^{\nu_j}-2^{\mu_j})=0$
with a suitable $t$. The linear sieve was employed in place of the
four dimensional vector sieve, consequently the quantity is
comparable to one fourth of those results in \cite{LiuLv,Li1}.

\section{Preliminary results}

The letter $\varepsilon$ denotes an arbitrary small positive
constant. The letter $N$ is a large integer and $L=(\log(N/\log
N))/\log 2$. To apply the circle method, we set
$$P=N^{\frac{1}{5}-\varepsilon},\ \ \ \ \ Q=L^{-14}N/P.$$
Define
\begin{equation}\label{e1}\mathcal{M}=\bigcup_{1\le q\le P}\bigcup_{\substack{1\le a\le q\\
(a,q)=1}}\mathcal{M}(q,a), \textrm{ and }
C(\mathcal{M})=[\frac{1}{Q}, 1+\frac{1}{Q}]\setminus
\mathcal{M},\end{equation} where
$$\mathcal{M}(q,a)=\{\alpha:|\alpha-\frac{a}{q}|\le
\frac{1}{qQ}\}.$$ Denote by $\mathcal{B}$ the interval
$[\sqrt{(1/4-\eta)N},\sqrt{(1/4+\eta)N}\,]$, where $\eta\in
(0,\frac{1}{10^{10}})$ is a constant.
Let\begin{align}\label{pro}T(\alpha)=\sum_{p\in \mathcal{B}}(\log
p)e(p^2\alpha),\ \ G(\alpha)=\sum_{4\le \nu\le
L}e(2^\nu\alpha).\end{align} Then we have
\begin{align*}R_k(N):=&\sum_{\substack{p^2_1+p_2^2+p_3^2+p_4^2+2^{\nu_1}+\cdots+2^{\nu_k}=N\,
\\ p_j\in \mathcal{B}(1\le j\le 4),
\ 4\le \nu_1,\ldots,\nu_k\le L}}\prod_{j=1}^{4}\log p_j
\\=&\int^1_0T^4(\alpha)G^k(\alpha)e(-\alpha N)d \alpha
\\=&\int_{\mathcal{M}}T^4(\alpha)G^k(\alpha)e(-\alpha N)d \alpha+\int_{C(\mathcal{M})
}T^4(\alpha)G^k(\alpha)e(-\alpha N)d \alpha.\end{align*}Let
\begin{equation}\label{e2}
C^\ast(q,a)=\sum_{\substack{m=1\\ (m,q)=1}}^{q}e(\frac{am^2}{q}),
\ \
B(n,q)=\sum_{\substack{a=1\\
(a,q)=1}}^{q}C^{\ast}(q,a)^4e(-\frac{an}{q}),
\end{equation}and
\begin{equation}\label{e3}
A(n,q)=\frac{B(n,q)}{\phi^4(q)},\ \
\mathfrak{S}(n)=\sum_{q=1}^\infty A(n,q).
\end{equation} For $n\equiv 4\pmod {24}$, we have
\begin{equation}\label{e4}
1\ll \mathfrak{S}(n) \ll (\log\log n)^{11}
\end{equation}and
\begin{equation}\label{e5}
\mathfrak{S}(n)=24\prod_{p>3}(1+A(n,p)).
\end{equation}
Define\begin{align*}
\mathfrak{I}(h)=\int_{-\infty}^\infty\big(\int_{\sqrt{1/4-\eta}}^{{\sqrt{1/4+\eta}}}e(x^2\beta)dx\big)^4e(-h\beta)d\beta.\end{align*}

On the major arcs, we quote
\begin{lemma}[Lemma 2.1\cite{LiuLv}]
\label{l1} For $2\le n\le N$, we have
$$\int_{\mathcal{M}}T^4(\alpha)e(-\alpha n)d \alpha=\mathfrak{S}(n)\mathfrak{I}(\frac{n}{N})N+O(\frac{N}{\log N}),$$
where $\mathfrak{S}(n)$ is given by (\ref{e3}).
\end{lemma}
The definition of $T(\alpha)$ in \eqref{pro} is slightly different
from that in \cite{LiuLv}, while the above result can be proved by
the same argument.

\section{An Application of the linear sieve}

Let
\begin{align}\label{defI}I=\int_0^{1}|T(\alpha)^4G(\alpha)^{14}|d\alpha.\end{align}
The purpose of this section is to obtain an upper bound for $I$ by
using the linear sieve. We first give an auxiliary lemma.
\begin{lemma}\label{lemmaJ}Let $$J=\sum_{
\substack{x_1^2+x_2^2=x_3^2+x_4^2
\\ 1\le x_1,x_2,x_3,x_4\le
P}}\tau(x_1)\tau(x_2)\tau(x_3)\tau(x_4),$$ where $\tau(n)$ denotes
the divisor function. Then we have
\begin{align}J\ll P^2(\log
P)^{14}.\end{align} \end{lemma} \begin{proof} One
has\begin{eqnarray*}J&=&\sum_{\substack{x_1^2+x_2^2=x_3^2+x_4^2\\x_1\not=x_3}}\tau(x_1)\tau(x_2)\tau(x_3)\tau(x_4)+\big(\sum_{x_1}\tau^2(x_1)
\big)^2 \\ &=:& J_{o}+J_{d}.
\end{eqnarray*}
The diagonal contribution $J_{d}$ is bounded by $P^2(\log P)^6$.
It suffices to prove $J_{o}\ll P^2(\log P)^{14}$. We have
\begin{eqnarray*}J_{o}&\le&\sum_{\substack{x_1^2+x_2^2=x_3^2+x_4^2\\x_1\not=x_3}}\tau^2(x_1)\tau^2(x_3)
\\
&=&2\sum_{x_1<x_3}\tau^2(x_1)\tau^2(x_3)\sum_{\substack{x_2,x_4\\(x_2-x_4)(x_2+x_4)=x_3^2-x_1^2}}1
\\ &\le &2\sum_{x_1<x_3}\tau^2(x_1)\tau^2(x_3)\tau(x_3^2-x_1^2)
\\ &\le &
2\big(\sum_{x_1<x_3}\tau^3(x_1)\tau^3(x_3)\big)^{2/3}\big(\sum_{x_1<x_3}\tau^3(x_3^2-x_1^2)\big)^{1/3}.\end{eqnarray*}
Note that $\sum_{1\le x_1<x_3 \le P}\tau^3(x_3^2-x_1^2)\le
\sum_{1\le a,b\le 2P}\tau^3(a)\tau^3(b)$. The desired result
follows from above easily.
\end{proof}
Let
\begin{align*}g(\beta)\ =\ \int_{\sqrt{1/4-\eta}}^{{\sqrt{1/4+\eta}}}e(x^2\beta)dx\ \textrm{ and
}\ g^{+}(\beta)\ =\
\int_{\sqrt{1/4-\eta-\eta^2}}^{{\sqrt{1/4+\eta+\eta^2}}}e(x^2\beta)dx.\end{align*}
Note that
\begin{align*}g(\beta),\ g^{+}(\beta)\ \ll \min\{1, |\beta|^{-1}\}.\end{align*}
 We introduce two integrals
\begin{align*}
\mathfrak{J}^{+}(h)=\int_{-\infty}^\infty g(\beta)^2
g(-\beta)g^{+}(-\beta) e(-h\beta)d\beta\end{align*} and
\begin{align*}
\mathfrak{J}(h)=\int_{-\infty}^\infty|g(\beta)|^4e(-h\beta)d\beta.\end{align*}
Note that $\mathfrak{J}^{+}(h)$ and $\mathfrak{J}(h)$ are
nonnegative constants depending on $\eta$. Moreover, one has
$\mathfrak{I}(1)\le \mathfrak{J}(0)\le \mathfrak{J}^{+}(0)\le
(1+O(\eta))\mathfrak{I}(1)$, where the $O$-constant is absolute.
Let \begin{align}\label{defSB} \mathbf{S}(h)\ =\
\prod_{p>2}\big(1+\frac{\mathbf{B}(p,h)}{(p-1)^4}\big),\end{align}
 where  \begin{align*}\mathbf{B}(p,h)\ =\
\sum_{\substack{a=1\\
(a,q)=1}}^{q}|C^\ast(p,a)|^4e(ah/p).\end{align*}
\begin{lemma}\label{lemma41}
Let $I$ be defined by (\ref{defI}). Then we have
\begin{align*}I\ \le \ 8(16+\varepsilon)\mathfrak{J}^{+}(0)N\sum_{h\not=0}r_7(h)&\mathbf{S}(h) \ +\
O(NL^{13}),\end{align*}where
\begin{align*}r_t(h)=\sum_{\substack{ 4\le \nu_j,\mu_j \le L
\\ \sum_{j=1}^{t}(2^{\nu_j}-2^{\mu_j})=h}}1,\ \ t\in \N.\end{align*}
\end{lemma}
\begin{proof}
Note that
\begin{align}I=\sum_{h\in \Z}r_7(h)\sum_{\substack{ p_j\in
\mathcal{B}
\\ p_1^2+p_2^2-p_3^2-p_4^2=h}}\prod_{j=1}^4\log p_j,\end{align}
Let us introduce a smooth function $w: \R^{+}\rightarrow [0,1]$,
which is supported on the interval
$[\sqrt{(1/4-\eta-\eta^2)},\sqrt{(1/4+\eta+\eta^2)}\,]$ and
satisfies $w(x)=1$ for all $x\in
[\sqrt{(1/4-\eta)},\sqrt{(1/4+\eta)}\,]$. It is clear that
\begin{align}\label{boundI}I\le I_{w}\log \sqrt{N},\end{align} where
\begin{align}I_{w}=\sum_{h\in\Z}r_7(h)\sum_{\substack{ p_1,p_2,p_3\in
\mathcal{B}
\\ p_1^2+p_2^2-p_3^2-p_4^2=h}}w(p_4/\sqrt{N})\prod_{j=1}^3\log p_j.\end{align}
Consider Rosser's weight $\lambda^{+}(d)$ of order
$D=N^{1/16-\varepsilon}$. Let $z=D^{1/2}$ and
$\Pi_{z}=\prod_{2<p<z}p$. Recalling the properties of Rosser's
weights, we know $|\lambda^{+}(d)|\le 1$,
$\sum_{d|(n,\Pi_{z})}\mu(d)\le
\sum_{d|(n,\Pi_{z})}\lambda^{+}(d)$, and
$$\lambda^{+}(d)=0\ \textrm{ if }\ \mu(d)=0\ \textrm{ or }\ d>D .$$ We
have
\begin{eqnarray*}I_{w} &\le &\sum_{h\in\Z}r_7(h)\sum_{\substack{ p_1,p_2,p_3\in
\mathcal{B},\ (y,\Pi_{z})=1
\\ p_1^2+p_2^2-p_3^2-y^2=h}}w(y/\sqrt{N})\prod_{j=1}^3\log p_j
\\&\le &\sum_{h\in\Z}r_7(h)\sum_{\substack{ p_1,p_2,p_3\in
\mathcal{B}
\\ p_1^2+p_2^2-p_3^2-y^2=h}}w(y/\sqrt{N})\big(\sum_{d|(y,\Pi_{z})}\lambda^{+}(d)\big)(\prod_{j=1}^3\log p_j)
\\&= &\sum_{d|\Pi_{z}}\lambda^{+}(d)
\sum_{h\in\Z}r_7(h)\sum_{\substack{ p_1,p_2,p_3\in \mathcal{B}
\\ p_1^2+p_2^2-p_3^2-d^2x^2=h}}w(dx/\sqrt{N})(\prod_{j=1}^3\log p_j)
\\ &:=&I_{w}^{+}.\end{eqnarray*}
Define
\begin{align*}f_d(\alpha)=\sum_{x}w(dx/\sqrt{N})e(d^2x^2\alpha),
\end{align*} and\begin{align*}
F(\alpha)=\sum_{d|\Pi_{z}}\lambda^{+}(d)f_d(\alpha).\end{align*}
Now $I_{w}^{+}$ can be represented as
\begin{align}I_{w}^{+}=\int_0^1T^2(\alpha)T(-\alpha)F(-\alpha)|G(\alpha)|^{14}d\alpha.\end{align}

Let \begin{equation}\label{Msmall}\mathfrak{M}=\bigcup_{1\le q\le N^{\eta}}\bigcup_{\substack{1\le a\le q\\
(a,q)=1}}\mathfrak{M}(q,a),\end{equation}where
$\mathfrak{M}(q,a)=\{\alpha: |q\alpha-a|\le N^{\eta}N^{-1}\}$.
Then we define
\begin{equation}\label{msmall}\mathfrak{m}=[N^{\eta}/N,1+N^{\eta}/N]\setminus\mathfrak{M}.\end{equation} So we
have
\begin{eqnarray}\label{Iplus}I_{w}^{+}=&&\sum_{h\not=0}r_7(h)
\int_{\mathfrak{M}}T^2(\alpha)T(-\alpha)F(-\alpha)e(-h\alpha)d\alpha\notag
\\
&+&r_7(0)\int_{\mathfrak{M}}T^2(\alpha)T(-\alpha)F(-\alpha)d\alpha
\\ &+&\int_{\mathfrak{m}}T^2(\alpha)T(-\alpha)F(-\alpha)|G(\alpha)|^{14}d\alpha.\notag\end{eqnarray}
We first consider the third integral on the right hand side of
(\ref{Iplus}). By Holder's inequality,
\begin{align*}&\int_{\mathfrak{m}}T^2(\alpha)T(-\alpha)F(-\alpha)|G(\alpha)|^{14}d\alpha
\\ \le &
\big(\int_0^1|T(\alpha)|^4d\alpha\big)^{3/4}\big(\int_{\mathfrak{m}}|F(\alpha)^4G(\alpha)^{56}|d\alpha\big)^{1/4}.\end{align*}
In light of Rieger's result \cite{R},
$\int_0^1|T(\alpha)|^4d\alpha \ll N\log^2 N$. Note that
\begin{align*}\int_{\mathfrak{m}}|F(\alpha)^4G(\alpha)^{56}|d\alpha
=&\sum_{h}r_{28}(h)\int_{\mathfrak{m}}|F(\alpha)^4|e(h\alpha)d\alpha
\\=&\sum_{h\not=0}r_{28}(h)\int_{\mathfrak{m}}|F(\alpha)^4|e(h\alpha)d\alpha+r_{28}(0)
\int_{\mathfrak{m}}|F(\alpha)^4|d\alpha.\end{align*} In view of
the work of Heath-Brown and Tolev \cite{HT} (see also
\cite{Tolev}), one has for $h\not=0$ that
\begin{eqnarray*}\int_{\mathfrak{m}}|F(\alpha)^4|e(h\alpha)d\alpha& \ll & N^{1-\delta},\end{eqnarray*}
where $\delta>0$ is a small constant depending on $\eta$.
Considering the underlying Diophantine equation, we have
\begin{eqnarray*}\int_{\mathfrak{m}}|F(\alpha)^4|d\alpha\le &\int_{0}^{1}|F(\alpha)^4|d\alpha &\le  J,\end{eqnarray*}
where $J$ is given by Lemma \ref{lemmaJ}. Hence
$\int_{\mathfrak{m}}|F(\alpha)^4|d\alpha \ll  NL^{14}$  and
\begin{eqnarray*}\int_{\mathfrak{m}}|F(\alpha)^4G(\alpha)^{56}|d\alpha
&\ll&NL^{42}.\end{eqnarray*}We conclude from above that
\begin{eqnarray*}\int_{\mathfrak{m}}T^2(\alpha)T(-\alpha)F(-\alpha)|G(\alpha)|^{14}d\alpha
&\ll&NL^{12}.\end{eqnarray*}The second integral in \eqref{Iplus}
can be handled similarly (actually easier). In particular, we have
\begin{eqnarray*}r_7(0)\int_{\mathfrak{M}}T^2(\alpha)T(-\alpha)F(-\alpha)d\alpha
&\ll&NL^{12}.\end{eqnarray*} Now we turn to the first integral in
\eqref{Iplus}, which is equal to
\begin{eqnarray*}\sum_{d|\Pi_{z}}\lambda^{+}(d)\int_{\mathfrak{M}}T^2(\alpha)T(-\alpha)f_d(-\alpha)e(-h\alpha)d\alpha.\end{eqnarray*}
Let us introduce
\begin{eqnarray*}\mathcal{S}_d(h)=\sum_{q=1}^{\infty}\frac{\mathcal{A}_d(q,h)}{q\phi^3(q)}
,\ \,\ \mathcal{S}(h)=\mathcal{S}_1(h),\end{eqnarray*} where
\begin{align*}\mathcal{A}_d(q,h)=\sum_{\substack{a=1\\
(a,q)=1}}^{q}C^\ast(q,a)^2C^\ast(q,-a)C(q,-ad^2)e(-ah/q)
\end{align*}
and
\begin{align*}C(q,a)=
\sum_{\substack{x=1}}^{q}e\big(\frac{ax^2}{q}\big).
\end{align*} Define $\Omega(d)=\frac{\mathcal{S}_d(h)}{\mathcal{S}(h)}$
provided that $\mathcal{S}(h)$ is nonzero, and $\Omega(d)=1$ if
$\mathcal{S}(h)$ is zero. Let \begin{align*}\ g^{+}_{w}(\beta)\ =\
\int_{\sqrt{1/4-\eta-\eta^2}}^{{\sqrt{1/4+\eta+\eta^2}}}w(x)e(x^2\beta)dx\end{align*}
and
\begin{align*}
\mathfrak{J}^{+}(h)\ =\ \int_{-\infty}^\infty g(\beta)^2
g(-\beta)g^{+}_{w}(-\beta) e(-h\beta)d\beta.\end{align*}

 The standard argument in the Waring-Goldbach problem implies
the asymptotic formula
\begin{eqnarray*}\int_{\mathfrak{M}}T^2(\alpha)T(-\alpha)f_d(-\alpha)e(-h\alpha)d\alpha
&=&\frac{\Omega(d)}{d}\mathcal{S}(h)\mathfrak{J}^{+}_{w}(h/N)N
\\ &&\ \ +O(d^{-1}N\log^{-A}N).\end{eqnarray*} In view of the properties
of Rosser's weights, we
have\begin{align*}\sum_{d|\Pi_{z}}&\lambda^{+}(d)\int_{\mathfrak{M}}T^2(\alpha)T(-\alpha)f_d(-\alpha)e(-h\alpha)d\alpha
\\ &\le\ \big(\Phi(2)+\varepsilon\big)\prod_{2<p<z}\big(1-\frac{\Omega(p)}{p}\big)\mathcal{S}(h)\mathfrak{J}^{+}_{w}(h/N)N+O(N\log^{-A}N),\end{align*}
where $\Phi(s)=2e^{\gamma}/s$ for $0<s\le 3$, and $\gamma$ is
Euler's constant. Note that $\Omega(2)=0$ when $h$ is even.
 Therefore we finally obtain
\begin{align*}I_{w}^{+}\ \le \ \sum_{h\not=0}r_7(h)&\big(\Phi(2)+\varepsilon\big)
\prod_{2<p<z}\big(1-\frac{\Omega(p)}{p}\big)\mathcal{S}(h)\mathfrak{J}^{+}_{w}(h/N)N
\ +\ O(NL^{12}).\end{align*} For $p>2$, one has
\begin{align*}(1-\frac{\Omega(p)}{p})(1-\frac{1}{p})^{-1}(1+\frac{\mathcal{A}_1(p,h)}{p(p-1)^3})\ =
\ 1+\frac{\mathbf{B}(p,h)}{(p-1)^4}.\end{align*}One also has
\begin{align*}1+\sum_{k=1}^\infty \frac{\sum_{a(2^k)^\ast}C^\ast(2^k,a)^2C^\ast(2^k,-a)C(2^k,-a)}{2^k\phi^3(2^k)}=4.\end{align*}
It is well-known that \begin{align*}\prod_{2\le
p<z}\big(1-\frac1p\big)=\frac{e^{-\gamma}}{\log
z}\Big(1+O(\frac{1}{\log z})\Big).\end{align*}
 Now we conclude that
\begin{align*}I_{w}^{+}\ \le& \
8(16+\varepsilon)(\log\sqrt{N})^{-1}\sum_{h\not=0}r_7(h)\mathbf{S}(h)\mathfrak{J}^{+}_{w}(h/N)N\
+\ O(NL^{12})
\\ \le& \
8(16+\varepsilon)\mathfrak{J}^{+}(0)N(\log\sqrt{N})^{-1}\sum_{h\not=0}r_7(h)\mathbf{S}(h)\
+\ O(NL^{12}).\end{align*} The desired conclusion now follows from
(\ref{boundI}) easily.
\end{proof}
\begin{lemma}\label{lemma33}
One has
\begin{align*}\int_{C(\mathcal{M})}|T(\alpha)^4G(\alpha)^{14}|d\alpha\ \le &\ 8(15+\varepsilon)
(1+O(\eta))\mathfrak{J}(0)N\sum_{h\not=0}r_7(h)\mathbf{S}(h) \\
&\ \  + O(NL^{13}).\end{align*}
\end{lemma}
\begin{proof}Recalling (\ref{e1}) and (\ref{msmall}), one has $C(\mathcal{M})\subseteq
\mathfrak{m}$ and
\begin{align*}\int_{C(\mathcal{M})}|T(\alpha)^4G(\alpha)^{14}|d\alpha\
\le \
\int_{\mathfrak{m}}|T(\alpha)^4G(\alpha)^{14}|d\alpha.\end{align*}
Note that
\begin{align*}\int_{\mathfrak{M}}|T(\alpha)^4G(\alpha)^{14}|d\alpha=\sum_{h\not=0}r_7(h)
\int_{\mathfrak{M}}|T(\alpha)^4|e(h\alpha)d\alpha+O(NL^{9}).\end{align*}
For $h\not=0$, the standard argument provides
\begin{align*}\int_{\mathfrak{M}}|T(\alpha)^4|e(h\alpha)d\alpha=8
&\mathbf{S}(h)\mathfrak{J}(h/N)N+O(NL^{-100}).\end{align*}Therefore
\begin{align*}\int_{\mathfrak{M}}|T(\alpha)^4G(\alpha)^{14}|d\alpha=8\sum_{h\not=0}r_7(h)
&\mathbf{S}(h)\mathfrak{J}(h/N)N+O(NL^{9}).\end{align*} Recalling
that $h\le NL^{-1}$, one has
\begin{align*}\int_{\mathfrak{M}}|T(\alpha)^4G(\alpha)^{14}|d\alpha=8\mathfrak{J}(0)(1+O(L^{-1}))N\sum_{h\not=0}r_7(h)
&\mathbf{S}(h)+O(NL^{9}).\end{align*}By Lemma \ref{lemma41}, we
obtain
\begin{align*}\int_{\mathfrak{m}}|T(\alpha)^4G(\alpha)^{14}|d\alpha=& \int_{0}^{1}\, -\, \int_{\mathfrak{M}}\
\ \\ \le& 8(15+\varepsilon)\sum_{h\not=0}r_7(h)
\mathbf{S}(h)(1+O(\eta))\mathfrak{J}(0)N+O(NL^{13}).\end{align*}
The desired conclusion is established.\end{proof}

\section{Numerical computations}
Throughout this section, we use $h$ to denote
$\sum_{j=1}^{7}(2^{u_j}-2^{v_j})$. For odd $q$, denote by
$\varrho(q)$ the smallest positive integer $\varrho$ such that
$2^{\varrho(q)}\equiv 1\pmod q$.

Define
\begin{align*}a(p)=&
\begin{cases}-(p+1)^2\ \ & \textrm{ if }\  p\equiv 3\pmod{4}
\\ 3p^2-2p-1\ \ & \textrm{ if }\  p\equiv 1\pmod{4}\end{cases},\end{align*}
and \begin{align*}b(p)=&
\begin{cases}(p-1)(p+1)^2\ \ & \textrm{ if }\  p\equiv 3\pmod{4}
\\(p-1)(p^2+6p+1)\ \ & \textrm{ if }\  p\equiv 1\pmod{4}\end{cases}
.\end{align*} Then we define the multiplicative function $c(d)$ by
$1+\frac{1}{c(p)}=\frac{1+\frac{b(p)}{(p-1)^4}}{1+\frac{a(p)}{(p-1)^4}}$,
where $d$ is square-free and $(30,d)=1$.
\begin{lemma}\label{constant} Let $c_0=\frac{25}{32}c_1+(\frac{3}{2}-\frac{25}{32})c_2$, where \begin{eqnarray*}c_1&:=&\sum_{\substack{
p|d\Rightarrow p>5}}
\frac{\mu^2(d)}{c(d)\varrho^{14}(3d)}\sum_{\substack{1\le
u_j,v_j\le \varrho(3d), 1\le j\le 7\\
3d|h}}1,\end{eqnarray*}and
\begin{eqnarray*}c_2&:=&\sum_{\substack{ p|d\Rightarrow p>5}}
\frac{\mu^2(d)}{c(d)\varrho^{14}(15d)}\sum_{\substack{1\le
u_j,v_j\le \varrho(15d), 1\le j\le 7\\
15d|h}}1.\end{eqnarray*} One has \begin{eqnarray*}c_0&<&
0.69.\end{eqnarray*}\end{lemma}\begin{proof}The proof is to follow
the lines in \cite{LLW1}. Set
$$\beta(d)=\big(\frac{1}{\varrho^{14}(3d)}\sum_{\substack{1\le
u_j,v_j\le \varrho(3d)\ 1\le j\le 7\\
3d|h}}1\big)^{-1}.$$ Then we have
\begin{align*}&c_1=\sum_{\substack{ p|d\Rightarrow p>5}}
\frac{\mu^2(d)}{c(d)}\int_{\beta(d)}^\infty\frac{d
x}{x^2}=\int_{2}^\infty\sum_{\substack{ p|d\Rightarrow p>5\\
\beta(d)\le x}}\frac{\mu^2(d)}{c(d)}\frac{d x}{x^2}.\end{align*}
Clearly $\beta(d)\ge \varrho(3d)$, so
$$\sum_{\substack{ p|d\Rightarrow p>5\\
\beta(d)\le x}}\frac{\mu^2(d)}{c(d)}\le \sum_{\substack{ p|d\Rightarrow p>5\\
\varrho(3d)\le x}}\frac{\mu^2(d)}{c(d)}.$$ Let
$m(x)=\prod_{e\le x}(2^e-1)$, and for $x\ge 3$ we have\begin{align*}\sum_{\substack{ p|d\Rightarrow p>5\\
\beta(d)\le x}}\frac{\mu^2(d)}{c(d)}\le \sum_{\substack{ p|d\Rightarrow p>5\\
3d|m(x)}}\frac{\mu^2(d)}{c(d)}&\le \prod_{\substack{p>5\\
p|m(x)}}(1+\frac{1}{c(p)}) \\ &\le \prod_{\substack{p>5}}\frac{1+\frac{1}{c(p)}}{1+\frac{1}{p-1}}\prod_{\substack{p>5\\
p|m(x)}}(1+\frac{1}{p-1}).\end{align*} It was proved in
\cite{LLW1} that $m(x)/\phi(m(x))\le e^{\gamma}\log x$ for $x\ge
9$. If $x\ge 9$, then
$$\sum_{\substack{ p|d\Rightarrow p>3\\
\beta(d)\le x}}\frac{\mu^2(d)}{c(d)}\le
\frac{8c_3}{15}e^{\gamma}\log x,$$where
$c_3=\prod_{\substack{p>5}}\frac{1+\frac{1}{c(p)}}{1+\frac{1}{p-1}}
\le 1.3904$. Let $M=40$. We have\begin{align*}c_1=&\int_{2}^{M}\sum_{\substack{ p|d\Rightarrow p>5\\
\beta(d)\le x}}\frac{\mu^2(d)}{c(d)}\frac{d x}{x^2}+\int^{\infty}_{M}\sum_{\substack{ p|d\Rightarrow p>5\\
\beta(d)\le x}}\frac{\mu^2(d)}{c(d)}\frac{d x}{x^2}\\\le
&\sum_{\substack{ p|d\Rightarrow p>5\\
\beta(d)<M}}\frac{\mu^2(d)}{c(d)}\int_{\beta(d)}^{M}\frac{d
x}{x^2}+\int^{\infty}_{M}\frac{8c_3}{15}e^{\gamma}\log x\frac{d
x}{x^2}\\=&\sum_{\substack{ p|d\Rightarrow p>3\\
\beta(d)<M}}\frac{\mu^2(d)}{c(d)}\big(\frac{1}{\beta(d)}-\frac{1}{M}\big)+\frac{8c_3}{15}e^{\gamma}\frac{1+\log
M}{M}.\end{align*} The constant $c_2$ can be handled in the
similar way. Then the numerical computations provide the desired
result.
\end{proof}

In the following lemma, the condition $(h)$ in $\sum_{(h)}$ means
the summation is taken over all $(u_1,\ldots,u_7,v_1,\ldots,v_7)$
satisfying $4\le u_j,v_j\le L$ and
$h=\sum_{j=1}^{7}(2^{u_j}-2^{v_j})\not= 0$.
\begin{lemma}
\label{constant2} Let
$\kappa(h)=\begin{cases}\frac{25+15(\frac{h}{5})}{32} & \textrm{
if }\ 5\nmid h \\ \frac{3}{2} & \textrm{ if }\ 5|h
\end{cases}$.
Then we have \begin{eqnarray}\label{inec}\sum_{\substack{(h)\\
h\equiv 0\pmod
3}}\kappa(h)\prod_{\substack{p>5\\
p|h}}\big(1+\frac{1}{c(p)}\big)&\le &\big(
\frac{25}{32}c_1+(\frac{3}{2}-\frac{25}{32})c_2+\varepsilon\big)L^{14}.\end{eqnarray}
\end{lemma}
\begin{proof}The left hand side of \eqref{inec} is equal to
\begin{align*}&\sum_{\substack{(h)\\ h\equiv 0\pmod 3 \\ 5\nmid h}}\
\frac{25+15(\frac{h}{5})}{32}\prod_{\substack{p>5\\
p|h }}\big(1+\frac{1}{c(p)}\big)+\frac{3}{2}\sum_{\substack{(h)\\
h\equiv 0\pmod {15}}}\prod_{\substack{p>5\\
p|h}}\big(1+\frac{1}{c(p)}\big)
\\=\ &\sum_{\substack{(h)\\ h\equiv 0\pmod 3 \\ 5\nmid h}}\
\frac{25}{32}\prod_{\substack{p>5\\
p|h }}\big(1+\frac{1}{c(p)}\big)+ \frac{3}{2}\sum_{\substack{(h)\\
h\equiv 0\pmod {15}}}\prod_{\substack{p>5\\
p|h}}\big(1+\frac{1}{c(p)}\big)+o(L^{14})
\\=\ &\frac{25}{32}\sum_{\substack{(h)\\ h\equiv 0\pmod 3 }}\
\prod_{\substack{p>5\\
p|h }}\big(1+\frac{1}{c(p)}\big)+(\frac{3}{2}-\frac{25}{32})\sum_{\substack{(h)\\
h\equiv 0\pmod {15}}}\prod_{\substack{p>5\\
p|h}}\big(1+\frac{1}{c(p)}\big) \\ &\ \ \ \ \ \ \ \ +o(L^{14})
\\=:&\frac{25}{32}\,\Sigma_1+(\frac{3}{2}-\frac{25}{32})\,\Sigma_2+o(L^{14}).\end{align*}
Let us consider $\Sigma_1$. One has
\begin{align*}\Sigma_1=&\sum_{\substack{(h)\\ h\equiv 0\pmod 3 }}\sum_{\substack{d|h \\
p|d\Rightarrow p>5}} \frac{\mu^2(d)}{c(d)}=\sum_{\substack{(h)\\ h\equiv 0\pmod 3 }}\sum_{\substack{d<N^\varepsilon \\ d|h \\
p|d\Rightarrow p>5}}
\frac{\mu^2(d)}{c(d)}+O(N^{-\varepsilon}) \\
\le & \sum_{\substack{d<N^\varepsilon  \\
p|d\Rightarrow p>5}} \frac{\mu^2(d)}{c(d)}\sum_{\substack{1\le
u_j,v_j\le L\\
3d|h}}1+O(N^{-\varepsilon}) \ =: \ \Sigma_1' \
+O(N^{-\varepsilon}).\end{align*} The sum $\Sigma_1'$ is bounded
by
\begin{align*}\le\ &\sum_{\substack{d<N^\varepsilon \\
p|d\Rightarrow p>5 \\ \varrho(3d)<L}}
\frac{\mu^2(d)}{c(d)}\sum_{\substack{1\le
u_j,v_j\le \varrho(3d)\\
3d|h}}\big(\frac{L}{\varrho(3d)}+O(1)\big)^{14}+\sum_{\substack{d<N^\varepsilon \\
p|d\Rightarrow p>5 \\ \varrho(3d)\ge L}}
\frac{\mu^2(d)}{c(d)}L^{13}
\\ \le\ &L^{14}
\sum_{\substack{d<N^\varepsilon  \\
p|d\Rightarrow p>5 }}
\frac{\mu^2(d)}{c(d)\varrho(3d)^{14}}\sum_{\substack{1\le
u_j,v_j\le \varrho(3d)\\
3d|h}}1+O(\varepsilon) L^{14}.\end{align*} Therefore $\Sigma_1\le
(c_1+\varepsilon) L^{14}$. Similarly, $\Sigma_2\le
(c_2+\varepsilon) L^{14}$. Now the desired conclusion is
established.
\end{proof}

\begin{lemma}\label{constant3}Let $\mathbf{S}(h)$ be given by \eqref{defSB}. Then we have
\begin{eqnarray*}\sum_{h\not=0}r_7(h)\mathbf{S}(h) &\le & 3c_0L^{14},\end{eqnarray*}
where $c_0$ is given by Lemma \ref{constant}.
\end{lemma}
\begin{proof}
Note that
\begin{align*}\mathbf{B}(p,h)=
\begin{cases}-(p+1)^2\ & \textrm{ if }\  p\equiv 3\pmod{4}\ \textrm{ and } p\nmid h
\\ -(p^2+6p+1)-4p(p+1)\big(\frac{h}{p}\big)\ & \textrm{ if }\  p\equiv 1\pmod{4}\ \textrm{ and } p\nmid h
\\(p-1)(p+1)^2\ & \textrm{ if }\  p\equiv 3\pmod{4}\ \textrm{ and } p| h
\\(p-1)(p^2+6p+1)\ & \textrm{ if }\  p\equiv 1\pmod{4}\ \textrm{ and } p| h\end{cases}
.\end{align*} Then we have
\begin{align*}
\mathbf{S}(h)\ \le&\ 3\,\widetilde{\kappa}(h)
\prod_{p>5}\big(1+\frac{a(p)}{(p-1)^4}\big)\prod_{\substack{5<p
\\
p|h}}\frac{1+\frac{b(p)}{(p-1)^4}}{1+\frac{a(p)}{(p-1)^4}},\end{align*}where
$\widetilde{\kappa}(h)=\kappa(h)$ if $3|h$ and zero otherwise. One
has
\begin{align*}c_4=&
\ \prod_{p>5}\big(1+\frac{a(p)}{(p-1)^4}\big) \le
0.9743.\end{align*}Therefore,
\begin{align*}
\mathbf{S}(h)\ \le&\ 3\, c_4\,\widetilde{\kappa}(h)
\prod_{\substack{5<p
\\ p|h}}\big(1+\frac{1}{c(p)}\big).\end{align*}
The conclusion now follows from Lemmas
\ref{constant}-\ref{constant2}.
\end{proof}
Let $$\Xi(N,k)=\{n\ge 2: n=N-2^{\nu_1}-\cdots-2^{\nu_k}, 4\le
\nu_1,\ldots,\nu_k \le L\}$$ for positive integer $k$.
\begin{lemma}
\label{l5.1}For $k\ge 35$ and $N\equiv 4\pmod{8}$, one has
$$\frac{1}{8}\sum_{\substack{n\in \Xi(N,k)\\ n\equiv 4\pmod{24}}}\mathfrak{S}(n)\ge 0.9NL^{k}.$$
\end{lemma}
\begin{proof} As shown in \cite{LiuLv}, for $p\equiv 1\pmod 4$,$$1+A(n,p)\ge 1-\frac{5p^2+10p+1}{(p-1)^4},$$while for $p\equiv 3\pmod 4$,
$$1+A(n,p)\ge 1-\frac{5p^2-2p+1}{(p-1)^4}.$$ We have the
numerical inequalities
$$\prod_{\substack{17\le p< p_{5000}\\ p\equiv 1\pmod 4}}(1-\frac{5p^2+10p+1}{(p-1)^4})
\prod_{\substack{17\le p< p_{5000}\\ p\equiv 3\pmod
4}}(1-\frac{5p^2-2p+1}{(p-1)^4})\ge 0.904923,$$where $p_r$ denotes
the $r$-th prime. Moreover
\begin{align*}\prod_{p\ge p_{5000}}(1+A(n,p))&\ \ge
\prod_{p\ge p_{5000}}(1-\frac{1}{(p-1)^2})^6 \\& \ \ge \prod_{m\ge
p_{5000}}(1-\frac{1}{(m-1)^2})^6 = (1-\frac{1}{p_{5000}-1})^6
\\ &\ >0.99994271.\end{align*}Thus
\begin{align*}&\prod_{p\ge 17}(1+A(n,p))\ge C_1:=0.904811.\end{align*}
Set $m_0=14$. Now we have
\begin{align*}\sum_{\substack{n\in \Xi(N,k)\\ n\equiv
4\pmod{24}}}\mathfrak{S}(n)\ \ge\ &24C_1\sum_{\substack{n\in
\Xi(N,k)\\ n\equiv 4\pmod{24}}}\prod_{3<p<m_0}(1+A(n,p))
\\=\ &24C_1\sum_{1\le
j\le q}\sum_{\substack{n\in \Xi(N,k)\\ n\equiv 4\pmod{24}\\
n\equiv j\pmod q}}\prod_{3<p<m_0}(1+A(n,p)) \\ =\ &
24C_1\sum_{1\le
j\le q}\prod_{3<p<m_0}(1+A(j,p))\sum_{\substack{n\in \Xi(N,k)\\ n\equiv 4\pmod{24}\\
n\equiv j\pmod q}}1,\end{align*}where $q=\prod_{3<p<m_0}p$.
Consider the inner sum, we have
\begin{align*}\mathcal{S}:=\sum_{\substack{n\in \Xi(N,k)\\ n\equiv 4\pmod{24}\\
n\equiv j\pmod
q}}1=&\big(\frac{L}{\varrho(3q)}+O(1)\big)^k\sum_{\substack{1\le
\nu_1,\ldots,\nu_k\le \varrho(3q)\\
2^{\nu_1}+\cdots+2^{\nu_k}\equiv a_j\pmod
{3q}}}1,\end{align*}where $a_j$ is the natural number in $[1,3q]$
satisfying $a_j\equiv 0\pmod 3$ and $a_j\equiv j\pmod q$. Note
that
\begin{align*}\mathcal{S}=&\ \big(\frac{L}{\varrho(3q)}+O(1)\big)^k\frac{1}{3q}\sum_{t=0}^{3q-1}e(\frac{ta_j}{3q})\big(\sum_{1\le
s\le \varrho(3q)}e(\frac{t2^s}{3q})\big)^k.\end{align*} We arrive
at\begin{align*}\mathcal{S}\ge&\
\big(\frac{L}{\varrho(3q)}+O(1)\big)^k\frac{1}{3q}\big(\varrho(3q)^k-(3q-1)
(\max)^k\big)
\\=&\
\frac{L^k}{3q}\big(1-(3q-1)
(\frac{\max}{\varrho(3q)})^k\big)+O(L^{k-1}),\end{align*}where
$$\max\ =\ \max\{|\sum_{1\le s\le
\varrho(3q)}e(\frac{j2^s}{3q})|:1\le j\le 3q-1\}.$$ Note that
$3q=15015, \varrho(3q)=60$. With the help of a computer, it is not
hard to check that
\begin{align*}\max\ =34.5\ldots <\ 34.6,\ \textrm{ and } (3q-1)(\frac{\max}{\varrho(3q)})^{35}<10^{-7}.\end{align*}Therefore
\begin{align*}\mathcal{S}\ge
\frac{(1-10^{-7})L^k}{3q}+O(L^{k-1}),\end{align*}and
\begin{align*}\sum_{\substack{n\in \Xi(N,k)\\ n\equiv 4\pmod{24}}}\mathfrak{S}(n)\ge &24C_1\sum_{j=1}^{
p}\prod_{3<p<m_0}(1+A(j,p))\frac{(1-10^{-7})L^k}{3q}+O(L^{k-1})
\\=&\frac{8C_1(1-10^{-7})L^k}{q}\prod_{3<p<m_0}\big(\sum_{j=1}^{
p}(1+A(j,p))\big)+O(L^{k-1}).\end{align*}Observing that
$$\sum_{j=1}^{
p}(1+A(j,p))=p+\frac{1}{(p-1)^4}\sum_{1\le a\le
p-1}C^4(p,a)\sum_{j=1}^{ p}e(-\frac{aj}{q})=p,$$ one has
$$\sum_{\substack{n\in \Xi(N,k)\\ n\equiv 4\pmod{24}}}\mathfrak{S}(n)\ \ge\, \ 8C_1(1-10^{-7})L^k+O(L^{k-1}).$$
The proof is completed since $L$ is sufficiently large.
\end{proof}

\section{Proof of Theorem \ref{thm}}
As in \cite{LiuLv}, it suffices to prove that large even integers
$N\equiv 4\pmod 8$ can be represented as the sum of four squares
of primes and $44$ powers of $2$, since for every even integer
$N$, there exist $u_1,u_2\in \{1,2,3\}$ such that
$N-2^{u_1}-2^{u_2}\equiv 4\pmod 8$. We set $k=44$. Let
$$E(\lambda)=\{\alpha\in (0,1]:\ |G(\alpha)|\ge \lambda L\}.$$
By Lemma 5.3 in \cite{LiuLv}, we know $$|E(0.887167)|\ll
N^{-\frac{3}{4}-10^{-10}}.$$ Let
$$\mathfrak{m}_1=C(\mathcal{M})\cap E(0.887167),$$ and
$$\mathfrak{m}_2=C(\mathcal{M})\setminus \mathfrak{m}_1.$$
Following the lines in \cite{LiuLv}, we have
\begin{align*}&\big|\int_{C(\mathcal{M})
}T^4(\alpha)G^k(\alpha)e(-\alpha N)d \alpha\big|\\ \le\ &
\int_{\mathfrak{m}_1}\big|T^4(\alpha)G^k(\alpha)\big|d
\alpha+\int_{\mathfrak{m}_2}\big|T^4(\alpha)G^k(\alpha)\big|d
\alpha\\ \le\ &
O(N^{1-\varepsilon})+(0.887167L)^{k-14}\int^1_0|T(\alpha)^4G(\alpha)^{14}|d
\alpha.
\end{align*}
By Lemma \ref{lemma33} and Lemma \ref{constant3},
\begin{align*}\big|\int_{C(\mathcal{M})
}\big|\le O(N)+(0.887167L)^{k-14}\ 3c_0\times
8(15+O(\eta)+\varepsilon)\mathfrak{J}(0)N L^{14}.
\end{align*}
On the major arcs, we have by Lemma \ref{l1},
$$\int_{\mathcal{M}}T^4(\alpha)G^k(\alpha)e(-\alpha N)d
\alpha=\sum_{\substack{ n\in \Xi(N,k) \\ n\equiv 8\pmod{24}}}
\mathfrak{S}(n)\mathfrak{I}(\frac{n}{N})N+O(NL^{k-1}).$$ For $n\in
\Xi(N,k)$, one has $\frac{n}{N}=1+O(L^{-1})$. Applying Lemma
\ref{l5.1}, we
get\begin{align*}&\int_{\mathcal{M}}T^4(\alpha)G^k(\alpha)e(-\alpha
N)d \alpha \\ = & \sum_{\substack{ n\in \Xi(N,k) \\ n\equiv
8\pmod{24}}}
\mathfrak{S}(n)\mathfrak{I}(1)(1+O(L^{-1}))N+O(NL^{k-1})\\ \ge&
0.9\times 8\,\mathfrak{I}(1)NL^k+O(NL^{k-1}).\end{align*}
Therefore we have
\begin{eqnarray*}R_{44}(N)&\ge&
8NL^{44}\big(0.9\mathfrak{I}(1)-(0.887167)^{30}
(45+\varepsilon)c_0(1+O(\eta))\mathfrak{J}(0)\big) \\ &>&0.00001
\mathfrak{J}(0)NL^{44}\end{eqnarray*}provided that $\eta$ is
sufficiently small. The proof of Theorem \ref{thm} is completed.


\end{document}